\documentclass[11pt]{amsart}
\usepackage{hyperref}

\usepackage{enumerate}

\usepackage[utf8]{inputenc}
\usepackage[T1]{fontenc}

\usepackage[mathscr]{eucal}

\usepackage[a4paper, hmargin={3cm,3cm}, centering]{geometry}

\everymath{\displaystyle}

\newtheorem{lemma}{Lemma}[section]
\newtheorem{theorem}[lemma]{Theorem}
\newtheorem*{theorem*}{Theorem}
\newtheorem{corollary}[lemma]{Corollary}

\newtheorem{proposition}[lemma]{Proposition}
\newtheorem*{proposition*}{Proposition}

\theoremstyle{remark}

\theoremstyle{definition}
\newtheorem*{definition*}{Definition}
\newtheorem*{conjecture*}{Conjecture}
\newtheorem*{remark*}{Remark}
\newtheorem*{claim*}{Claim}

\newcommand{\C}{{\mathbb C}}
\newcommand{\E}{{\mathbb E}}

\newcommand{\N}{{\mathbb N}}

\newcommand{\Q}{{\mathbb Q}}
\newcommand{\R}{{\mathbb R}}

\newcommand{\Z}{{\mathbb Z}}

\newcommand{\norm}[1]{\left\Vert #1\right\Vert}
\newcommand{\nnorm}[1]{\lvert\!|\!| #1|\!|\!\rvert}

\begin{document}

\title[Integer part polynomial correlation sequences]{Integer part polynomial correlation sequences}

\author{Andreas Koutsogiannis}
\address[Andreas Koutsogiannis]{The Ohio State University, Department of mathematics, Columbus, Ohio, USA} \email{koutsogiannis.1@osu.edu}

\begin{abstract}
Following an approach presented by N.~Frantzikinakis, we prove that any multiple correlation sequence, defined by invertible measure preserving actions of  commuting  transformations with integer part polynomial iterates,  is the sum of a nilsequence and an error term, small in uniform density. As an intermediate result, we show that multiple ergodic averages with iterates given by the integer part of real valued polynomials converge in the mean. Also,  we show that under certain assumptions the limit is zero. An important role in our arguments  plays a transference principle, communicated to us by M.~Wierdl, that enables to deduce results for  $\mathbb{Z}$-actions from results for flows.
\end{abstract}



\maketitle

\section{Introduction and main results}

  In this paper we study the structure of sequences of the form 
\begin{equation}\label{E:integer}
a(n)=\int f_0\cdot (\prod_{i=1}^\ell T_i^{[p_{i,1}(n)]}) f_1\cdot \ldots \cdot (\prod_{i=1}^\ell T_i^{[p_{i,m}(n)]}) f_m\, d\mu, \quad n\in \N,
\end{equation}
which we call \emph{integer part polynomial correlation sequences}, where $[\cdot]$ denotes the integer part, $T_1,\ldots, T_\ell\colon X\to X$ are invertible commuting measure preserving transformations on a probability space $(X,\mathcal{X},\mu),$ $f_0,f_1,\ldots, f_m\in L^\infty(\mu)$ and $p_{i,j}\in \R[t]$ are real valued polynomials for all $1\leq i\leq \ell,$ $1\leq j\leq m.$ Even if it's not stated, without loss of generality, for the bounded functions $f_i,$ we will always assume that $\norm{f_i}_\infty\leq 1$ for all $i.$

\begin{definition*}
We call the setting $(X,\mathcal{X},\mu,T_1,\ldots,T_\ell)$  a \emph{system}, where $T_1,\ldots, T_\ell\colon X\to X$ are invertible commuting measure preserving transformations on the probability space $(X,\mathcal{X},\mu).$ 
\end{definition*}

 Any sequence of the form \eqref{E:integer} is a special case of a multiple correlation sequence, i.e., sequence of the form
\begin{equation}\label{E:mcs}
\int f_0\cdot T_1^{n_1}f_1\cdot \ldots \cdot T_\ell^{n_\ell}f_\ell\, d\mu,
\end{equation} with $n_1,\ldots,n_\ell\in \Z.$ 
Results on the structure and the limiting behaviour of averages of multiple correlation sequences is a central problem in ergodic  (Ramsey) theory. Although determining the precise structure of such sequences is an important open problem in the area, in recent years a lot of progress has been made.

\medskip

 In order to state some relevant results, we recall the notion of an $\ell$-step nilsequence. 
  
 \begin{definition*}[\cite{BHK05}]
For $\ell\in \N,$  an  \emph{$\ell$-step nilsequence} is a sequence of the form $(F(g^n\Gamma))$, where $F\in C(X)$, $X=G/\Gamma,$ $G$ is an  $\ell$-step nilpotent Lie group, $\Gamma$ is a discrete cocompact subgroup and
  $g\in G$. A  \emph{$0$-step nilsequence}
is a constant sequence.
 \end{definition*}

 In the special case of a single ergodic transformation (i.e., all the invariant sets under the transformation have measure either $0$ or $1$) where $T_i=T^i,$ $i=1,\ldots, \ell,$ 
Bergelson, Host and Kra proved the following result:

\begin{theorem*}[\mbox{\cite[Theorem 1.9]{BHK05}}]\label{T:0}
 For $\ell\in \N,$ let $(X,\mathcal{X},\mu, T)$ be an ergodic system  and    $f_0,f_1,$ $\ldots,f_\ell\in L^{\infty}(\mu).$ Then  we have the
 decomposition
$$
\int f_0\cdot T^{n}f_1\cdot \ldots \cdot T^{\ell n}f_\ell\ d\mu=\mathcal{N}(n)+e(n), \quad  n\in \N, $$
where
\begin{enumerate}
\item $(\mathcal{N}(n)) \text{ is a uniform limit of}\;\ell\text{-step nilsequences
with} \norm{\mathcal{N}}_{\infty}\leq 1;$

\item  $\lim_{N-M\to \infty} \frac{1}{N-M}\sum_{n=M}^{N-1} |e(n)|^2=0.$
\end{enumerate}
\end{theorem*}

The polynomial iterates version of this result is due to Leibman (in \cite{L10}). Also, the same author (in \cite{L14}), proved the result of Bergelson, Host and Kra without the ergodicity assumption.

\medskip

All these results depend on the theory of  characteristic factors, a tool that in a more general setting, say for correlation sequences involving actions of commuting transformations, proved to be extremely complex. 

\medskip
 
Quite recently, Frantzikinakis (in \cite{F}) showed that this problem can be settled, since he showed, avoiding the use of the theory of characteristic factors, that
modulo error terms, which are small in uniform density, correlation sequences of actions of commuting transformations are nilsequences (\mbox{\cite[Theorem 1.3]{F}}). More specifically, using the convergent result of Walsh from \cite{W12} and tools from \cite{HK99} and \cite{HK09},  he proved:

\begin{theorem*}[\mbox{\cite[Theorem 1.2]{F}}]\label{T:1'}
 Let   $\ell,m\in \N$ and  $p_{i,j}\in \Z[t],$ $1\leq i\leq \ell, 1\leq j\leq m,$ be polynomials.
Then  there exists $k\in \N,$ $k=k(\ell,m,\max{\deg(p_{i,j})})$,
such that for every system
  $(X,\mathcal{X},\mu, T_1,\ldots, T_\ell)$, functions
 $f_0,f_1,\ldots,f_m\in L^{\infty}(\mu)$ and $\varepsilon>0$, we have
\begin{equation}\label{E:complicated}
\int f_0\cdot (\prod_{i=1}^\ell T_i^{p_{i,1}(n)})
f_1\cdot \ldots \cdot (\prod_{i=1}^\ell T_i^{p_{i,m}(n)})
f_m\, d\mu =\mathcal{N}(n)+e(n), \quad n\in \N,
\end{equation}
where
\begin{enumerate}
\item $(\mathcal{N}(n))$ is a $k$-step nilsequence with $\norm{\mathcal{N}}_{\infty}\leq 1$;

\item \label{E:2} $\displaystyle \lim_{N-M\to \infty} \frac{1}{N-M}\sum_{n=M}^{N-1} |e(n)|^2\leq \varepsilon.$
\end{enumerate}
\end{theorem*}

The arguments that give this result (see \cite{F}) 
 focus on some distinctive properties that correlation sequences as  in \eqref{E:complicated} satisfy. Actually, \cite{F} deals with a more general setting providing a sufficient condition (\mbox{\cite[Theorem 1.3]{F}}) in order a multiple correlation sequence to have the required decomposition in nil$+$nul parts. So, in order to obtain the previous result, 
 one 
 proves that sequences of the form \eqref{E:complicated} satisfy the 
  sufficient condition.

\medskip

In this article, we get an analogous result for the respective integer part polynomial correlation sequences.
 Namely, we prove the following:

\begin{theorem}\label{T:1*}
Let   $\ell,m\in \N$ and  $p_{i,j}\in \R[t]$, $1\leq i\leq \ell, 1\leq j\leq m,$ be polynomials$.$
Then  there exists $k\in \N$, $k=k(\ell,m,\max{\deg(p_{i,j})})$,
such that for every system
  $(X,\mathcal{X},\mu,$ $ T_1,\ldots, T_\ell)$, functions
 $f_0,f_1,\ldots,f_m\in L^{\infty}(\mu)$ and $\varepsilon>0$, we have
\begin{equation}\label{E:complicated*}
\int f_0\cdot (\prod_{i=1}^\ell T_i^{[p_{i,1}(n)]})
f_1\cdot \ldots \cdot (\prod_{i=1}^\ell T_i^{[p_{i,m}(n)]})
f_m\, d\mu =\mathcal{N}(n)+e(n), \quad n\in \N,
\end{equation}
where
\begin{enumerate}
\item $(\mathcal{N}(n))$ is a $k$-step nilsequence with $\norm{\mathcal{N}}_{\infty}\leq 1$;

\item \label{E:2} $\displaystyle \lim_{N-M\to \infty} \frac{1}{N-M}\sum_{n=M}^{N-1} |e(n)|^2\leq \varepsilon.$
\end{enumerate}
\end{theorem}

In order to prove this result, we will also give, in Theorem~\ref{T:2*}, a sufficient condition, analogous to the one in Theorem~1.3 of \cite{F}, in order to have the decomposition in nil$+$nul terms of Theorem~\ref{T:1*}. Theorem~\ref{T:1*} then will follow by showing that sequences of the form \eqref{E:complicated*} satisfy the aforementioned sufficient condition.

\medskip

Note that Theorem~\ref{T:1*} is a result in a more general scheme. It can be considered as the first step towards the understanding of the structure of correlation sequences with iterates given by generalized polynomials (see definition below).

\begin{definition*}[\cite{BMc}]\label{D:GP}
We denote by $\mathcal{G}$ the smallest family of functions
$\N\to \Z$ containing $\Z[n]$ that forms an algebra under addition and multiplication and
having the property that for every $f_1,\ldots, f_r \in \mathcal{G}$ and $c_1, \ldots, c_r \in \R,\; [\sum^r_{i=1} c_i f_i] \in \mathcal{G}$ (i.e., $\mathcal{G}$ contains all the functions that can be obtained from regular polynomials with
the help of the floor function and the usual arithmetic operations).
The members of $\mathcal{G}$ are called {\em generalized polynomials}.
\end{definition*}

In the special case of Theorem~\ref{T:1*}, where the polynomials are linear, and so of the form $p_{i,j}(n)=a_{i,j}n$ for $a_{i,j}\in\R,$ $1\leq i\leq \ell,$ $1\leq j\leq m,$ the next result 
gives more precise information about the dependence of $k$ on the other parameters.

\begin{theorem}\label{T:s}
For $\ell, m\in \N$ let $(X,\mathcal{X},\mu, T_1,\ldots, T_\ell)$ be a system, $a_{i,j}\in \R,$ $1\leq i\leq \ell,$ $1\leq j\leq m$    and $f_0,f_1,$ $\ldots,f_m\in L^{\infty}(\mu).$ Then, for every $\varepsilon>0,$ we have the
 decomposition
\begin{equation}\label{E:mulcom}
\int f_0\cdot (\prod_{i=1}^\ell T_i^{[a_{i,1} n]})f_1\cdot \ldots \cdot (\prod_{i=1}^\ell T_i^{[a_{i,m} n]})f_m\ d\mu=\mathcal{N}(n)+e(n), \quad  n\in \N,
\end{equation}
where
\begin{enumerate}
\item $(\mathcal{N}(n))$ is an $m$-step nilsequence with $\norm{\mathcal{N}}_{\infty}\leq 1$;

\item \label{E:2} $\lim_{N-M\to \infty} \frac{1}{N-M}\sum_{n=M}^{N-1} |e(n)|^2\leq \varepsilon$.
\end{enumerate}
\end{theorem}

We will also give an application of the previous results.

\medskip

 For $k\in \N,$ we consider the following  subsets of $\ell^\infty(\N)$:
$$
\mathcal{A}_k:=\Big\{(\psi(n))  \colon \psi \text{ is a } k\text{-step nilsequence}\Big\};
$$
$$
\mathcal{B}_k:=\Big\{\int f_0\cdot  \prod_{i=1}^{k}T^{(\ell_i-\ell_{k+1})n}f_i \, d\mu\colon  (X,\mathcal{X}, \mu,T) \text{ is a system}, f_i\in L^\infty(\mu) \text{ and }  \ell_i=\frac{(k+1)!}{i} \Big\};
$$

$$
\mathcal{C}_k:=\Big\{\int f_0\cdot (\prod_{i=1}^k T_i^{[a_{i,1} n]})f_1\cdot \ldots \cdot (\prod_{i=1}^k T_i^{[a_{i,k} n]})f_k\, d\mu\colon (X,\mathcal{X}, \mu, T_1,\ldots, T_k) \text{ is a system,}$$ $$ a_{1,1}, \ldots, a_{k,1}, a_{1,2}, \ldots, a_{k,2},\ldots, a_{1,k},\ldots, a_{k,k} \in \R \text{ and } f_i\in L^\infty(\mu)\Big\}.\;\;
$$

It is proven in \cite{F} that the sets $\mathcal{A}_k$ and $\mathcal{B}_k$ are linear subspaces of $\ell^\infty(\N).$ By the same reasoning (with the space $\mathcal{B}_k$), we have that $\mathcal{C}_k$ is a linear subspace of $\ell^\infty(\N)$ as well.

\medskip

Also, from Theorem~1.4 in \cite{F}, we have that, modulo sequences small in uniform density,  the two subspaces
$\mathcal{A}_k$ and $\mathcal{B}_k$
 coincide. Namely, we have that 
 
 $$
\overline{\mathcal{A}_{k}}^{\norm{\cdot}_2}=\overline{\mathcal{B}_k}^{\norm{\cdot}_2},
$$

where $\norm{\cdot}_2$ is the seminorm on $\ell^\infty(\N)$ defined  by
\begin{equation}\label{E:seminorm}
\norm{a}_2^2:=\limsup_{N-M\to \infty} \frac{1}{N-M}\sum_{n=M}^{N-1}|a(n)|^2.
\end{equation}


Using Theorem~\ref{T:s}, we will prove the following:

\begin{theorem} \label{T:a}
For every $k\in \N$ we have

$$
\overline{\mathcal{A}_{k}}^{\norm{\cdot}_2}=\overline{\mathcal{B}_k}^{\norm{\cdot}_2}=\overline{\mathcal{C}_k}^{\norm{\cdot}_2}.
$$
\end{theorem}

In order to prove Theorem~\ref{T:1*}, as an intermediate step, we also prove the following result on the existence of the integer part polynomial multiple mean convergence:

\begin{theorem}\label{T:3*}
For $\ell, m\in \N$ let $(X,\mathcal{X},\mu,T_1,\ldots, T_\ell)$ be a system, $p_{i,j}\in\R[t],$ polynomials, $1\leq i\leq \ell,$ $1\leq j\leq m$ and  $f_1,\ldots,f_m\in L^{\infty}(\mu).$ The averages
 \begin{equation}\label{E:lou}
 \frac{1}{N-M}\sum_{n=M}^{N-1}
(\prod_{i=1}^\ell T_i^{[p_{i,1}(n)]})f_1\cdot\ldots\cdot(\prod_{i=1}^\ell T_i^{[p_{i,m}(n)]})f_m
\end{equation}  converge in $L^2(\mu)$ as $N-M\to \infty.$
\end{theorem}

Also, in some special cases of Theorem~\ref{T:3*}, via the theory of characteristic factors (equivalently, via the seminorms $\nnorm{\cdot}_k$) using Theorem~1.2 and Proposition~5.1 from \cite{CFH}, we prove convergence to 0 for the previous averages. More specifically, we prove the following (for the definitions, see Section 2):

\begin{theorem}\label{T:F1}
For $\ell\in \N$ let $(X,\mathcal{X},\mu,T_1,\ldots, T_\ell)$ be a system, $a_{1},\ldots,a_\ell\in\R,$ be non-zero real numbers, $r_1,\ldots,r_\ell\in \N,$ be pairwise distinct positive integers  and  $f_1,\ldots,f_\ell\in L^{\infty}(\mu).$ Then there exists $k=k(\ell,\max{r_i})\in \N$ such that if $\nnorm{f_i}_{k,\mu,T_i}=0$ for some $1\leq i\leq \ell,$ then the averages
$$ \frac{1}{N-M}\sum_{n=M}^{N-1}
T_1^{[a_1 n^{r_1}]} f_1\cdot\ldots\cdot T_\ell^{[a_\ell n^{r_\ell}]} f_\ell$$  converge to $0$ in $L^2(\mu)$ as $N-M\to \infty.$
\end{theorem}

Moreover, we have the following result that applies to a larger class of polynomial iterates:

\begin{theorem}\label{T:F2}
For $\ell\in \N$ let $(X,\mathcal{X},\mu,T_1,\ldots, T_\ell)$ be a system, $p_{1},\ldots,p_\ell\in\R[t],$ be non-constant polynomials with distinct degrees and highest degree $d=\deg(p_1)$ and  $f_1,\ldots,f_\ell\in L^{\infty}(\mu).$ Then there exists $k=k(\ell,d)\in \N$ such that if $\nnorm{f_1}_{k,\mu,T_1}=0,$ then the averages
$$ \frac{1}{N-M}\sum_{n=M}^{N-1}
T_1^{[p_{1}(n)]} f_1\cdot\ldots\cdot T_\ell^{[p_{\ell}(n)]} f_\ell$$  converge to $0$ in $L^2(\mu)$ as $N-M\to \infty.$
\end{theorem}

Actually, we remark at this point that a more general result is proven in Proposition~\ref{P:F3} which implies Theorem~\ref{T:F2}. 

\medskip

Since for weak mixing systems $(X,\mathcal{X},\mu,T)$ we have that $\int fd\mu =0$ implies $\nnorm{f}_{k,\mu,T}=0$ for every $k\in \N,$ we 
deduce: 

\begin{corollary}\label{C:F10}
For $\ell \in \N$ let $(X,\mathcal{X},\mu,T_i),$ $1\leq i\leq \ell,$ commuting weak mixing systems,  $p_1,\ldots, p_\ell \in \R[t]$ be non-constant polynomials with distinct degrees 
and $f_1,\ldots,f_\ell\in L^\infty(\mu).$ Then the averages 
$$ \frac{1}{N-M}\sum_{n=M}^{N-1}
T_1^{[p_{1}(n)]} f_1\cdot\ldots\cdot T_\ell^{[p_{\ell}(n)]} f_\ell$$  converge to $\prod_{i=1}^\ell\int f_i \; d\mu$ in $L^2(\mu)$ as $N-M\to \infty.$
\end{corollary}

\begin{remark*}
It is an open problem if the same result (of Corollary~\ref{C:F10}) is true when the non-constant polynomials $p_i$ are essentially distinct (i.e., $p_i-p_j$ is non-constant for $i\neq j$).
\end{remark*}



As Theorem~\ref{T:1*}, Theorem~\ref{T:3*} is a first result towards establishing mean convergence for the more general case where the iterates are given by generalized polynomials. For these two results, i.e., dealing with integer part polynomial iterates, the method of proof relies on a respective convergence  result for flows. This technique was first used in \cite{EL} by E.~Lesigne, in order to prove that when a sequence of real positive numbers is good for the single term pointwise ergodic theorem, then the respective sequence of its integer parts is also good 
(see also \cite{BJW}). This method later adapted by M.~Wierdl (in \cite{Mate})  to deal with multiple term averages (see Theorem~\ref{T:M} below).

\medskip

\noindent Before we close this section, we state the following conjecture:

\begin{conjecture*}
Theorems~\ref{T:1*} and ~\ref{T:3*} hold if $p_{i,j}$ are generalized polynomials.
\end{conjecture*} 

Note that for Theorem~\ref{T:1*}, i.e., decomposition in the form nil$+$nul, the conjecture in its generality is completely open. As for Theorem~\ref{T:3*}, only the one-term case is known due to Bergelson and Leibman (\cite{BL}). Even the case where all the transformations are equal is not known.

\medskip

\noindent {\bf Notation.}
  We denote by $\N$ the set of positive integers.
If $(a(n))$ is a bounded sequence, we denote by  $\displaystyle \limsup_{N-M\to \infty} \Big|\frac{1}{N-M}\sum_{n=M}^{N-1}a(n) \Big|$ the limit 
$
\displaystyle \lim_{N\to \infty} \sup_{M\in \N} \Big|\frac{1}{N}\sum_{n=M}^{M+N-1}a(n)\Big|
$ (this limit exists by subadditivity). For a measurable function $f$ on a measure space $X$ with a transformation $T:X\to X,$ we denote with $Tf$ the composition $f\circ T.$ Given transformations $T_i:X\to X,$ $1\leq i\leq \ell,$ we denote with $\prod_{i=1}^\ell T_i$ the composition $T_1\circ \cdots\circ T_\ell.$

\subsection*{ Acknowledgements.} I would like to express my indebtedness to N.~Frantzikinakis who introduced to me this problem and for his in-depth suggestions throughout the writing procedure of this article. These suggestions corrected mistakes from the original version of this paper and led me to additional results. I would also like to thank M.~Wierdl for making an unpublished note of his, available to me. The proof of Theorem~\ref{T:M} below is based on this note. Finally, I wish to thank V. Bergelson for helpful discussions during the preparation of this paper.

\section{Definitions and Main Ideas}

\subsection{Anti-uniformity and Regularity}
 In order to prove Theorem~\ref{T:1*}, we follow the arguments of \cite{F}. For reader's convenience, in this subsection, we repeat most of the part of Section 2 from \cite{F}. First we recall the notion of the uniformity seminorms (a slight variant of the uniformity seminorms defined by B.~Host and B.~Kra in \cite{HK09}).

\begin{definition*}[\cite{F}]
Let $k\in \N$ and $a\colon \N\to \C$ be a bounded sequence.

\begin{enumerate}
\item
Given a sequence of intervals  ${\bf I}=(I_N)$   with  lengths tending to infinity,
we say that the   sequence $(a(n))$ is {\em distributed
 regularly along ${\bf I}$}  if the limit
 $$
 \lim_{N\to \infty} \frac{1}{|I_N|}\sum_{n\in I_N} a_1(n+h_1)\cdot\ldots\cdot a_r(n+h_r)
$$
exists for every $r\in \N$ and  $h_1,\ldots, h_r\in \N$, where $a_i$ is either $a$ or $\bar{a}.$

\item If ${\bf I}$  is as in (i)  and $(a(n))$ is  distributed
 regularly along ${\bf I}, $ we define inductively
$$\norm{a}_{{\bf I}, 1}:=
\lim_{N\to \infty} \Big|\frac{1}{|I_N|}\sum_{n\in I_N} a(n)\Big|;
$$
and for $k\geq 2$ (one  can  show as in \mbox{\cite[Proposition 4.3]{HK09}} that the next limit exists)
$$
\norm{a}_{{\bf I}, k}^{2^{k}} :=\lim_{H\to \infty} \frac{1}{H}\sum_{h=1}^H \norm{\sigma_ha\cdot \bar{a}}^{2^{k-1}}_{{\bf I}, k-1},
$$
where $\sigma_h$ is the shift transformation defined by $(\sigma_ha)(n):=a(n+h).$
\item If $(a(n))$ is a bounded sequence we let
$$
\norm{a}_{U_k(\N)}:=\sup_{{\bf I}}\norm{a}_{{\bf I}, k},
$$
where the sup is taken over all sequences of intervals ${\bf I}$  with lengths tending to infinity along which
the sequence  $(a(n))$ is  distributed
 regularly$.$
\end{enumerate}
\end{definition*}

Next, we recall the notions of $k$-anti-uniformity and $k$-regularity from \cite{F}. These notions play an important role to the decomposition of multiple correlation sequences in the form nil+nul. We will adapt them, dealing with the notion of regularity and a notion similar to anti-uniformity for our setting (i.e., the notion of weak-anti-uniformity).

\begin{definition*}[\cite{F}]
Let $k\in \N.$ We say that the bounded sequence $a\colon \N\to \C$ is
\begin{enumerate}
\item
 {\em $k$-anti-uniform}  if
 there exists
 $C:=C(k,a)$ such that
 $$
 \limsup_{N-M\to
  \infty} \Big|\frac{1}{N-M}\sum_{n=M}^{N-1} a(n)b(n)\Big|\leq C \norm{b}_{U_{k}(\mathbb{N})}$$ for every $b\in \ell^{\infty}(\N).$

\item  {\em $k$-regular} if
 the limit
 $$
 \lim_{N-M\to \infty} \frac{1}{N-M}\sum_{n=M}^{N-1} a(n)\psi(n)
 $$ exists for every $(k-1)$-step nilsequence $(\psi(n)).$
\end{enumerate}
\end{definition*}

 It turns out that the previous properties ($k$-anti-uniformity and $k$-regularity) give a sufficient condition on the required decomposition of sequences of the form \eqref{E:complicated} (\mbox{\cite[Theorem 1.3]{F}}). 

\medskip

 In our case though, for sequences of the form \eqref{E:integer}, we cannot prove that they are $k$-anti-uniform, but we can prove something weaker. Namely, in Theorem~\ref{T:M2} we will prove that these sequences are $k$-weak-anti-uniform (see definition below).

\begin{definition*}
Let $k\in \N.$ We say that the bounded sequence $a:\N\to\C$ is
{\em $k$-weak-anti-uniform}  if for every $0<\delta<1$
 there exists a constant 
 $C_{\delta}:=C_\delta(k,a)$ and a term $c_\delta,$ with $c_{\delta}\to 0$ as $\delta\to 0^+,$  such that for every $b\in \ell^{\infty}(\N)$
 $$
 \limsup_{N-M\to
  \infty} \Big|\frac{1}{N-M}\sum_{n=M}^{N-1} a(n)b(n)\Big|\leq C_\delta \norm{b}_{U_{k}(\mathbb{N})}+c_\delta.$$  
\end{definition*}

  The same proof of Theorem~1.3 in \cite{F}, works in our case as well, namely, we have the following (we only sketch the proof, for details, see the proof of Theorem~1.3 from \cite{F}): 

\begin{theorem}[\mbox{\cite[Theorem 1.3]{F}}]\label{T:2*}
For $k\in \N$ let $a\colon \N\to \C$ be a sequence with $\norm{a}_\infty\leq 1$ that is $k$-weak-anti-uniform and $k$-regular$.$
Then, for every $\varepsilon>0,$ we have the decomposition
$$
a(n)=\mathcal{N}(n)+e(n), \quad n\in \N,
$$
where
\begin{enumerate}
\item \label{E:1a} $(\mathcal{N}(n))$ is a $(k-1)$-step nilsequence with $\norm{\mathcal{N}}_{\infty}\leq 1$;

\item \label{E:1b} $\displaystyle \lim_{N-M\to \infty} \frac{1}{N-M}\sum_{n=M}^{N-1} |e(n)|^2\leq \varepsilon.$
\end{enumerate}
\end{theorem}

\begin{proof}(\mbox{\cite[Theorem 1.3]{F}})  
 We first remark that the limit
 $$
  \lim_{N-M\to \infty} \frac{1}{N-M}\sum_{n=M}^{N-1} |a(n)|^2 \quad \text{ exists}.
$$

Let
 $
Y:=\Big\{(\psi(n)) \colon \psi \text{ is a } (k-1)\text{-step nilsequence}\Big\}
$
and
$
X:=\text{span}\{{Y,a\}}.
$

On $X\times X$ we define the bilinear form
$$
\langle f, g \rangle:=\lim_{N-M\to \infty} \frac{1}{N-M}\sum_{n=M}^{N-1} f(n)\overline{g}(n).
$$
  Since the limit exists for $f,g\in X$, this bilinear form induces the  seminorm
$$
\norm{f}_2:=\sqrt{\langle f, f \rangle}
$$
(note that this is the restriction, on the space $X,$ of the seminorm defined in \eqref{E:seminorm}).

Let $\varepsilon>0$. Then, there exists $\delta_0>0$ such that for all $(b(n))\in \ell^\infty(\N)$ we have $\limsup_{N\to\infty} \Big| \frac{1}{|I_N|}\sum_{n\in I_N}a(n)b(n) \Big| \leq C_{\delta_0}\norm{b}_{U_k(\N)}+\frac{\varepsilon}{4}.$ We can assume that $ C_{\delta_0}\geq 1.$

\noindent For $d:=\inf\{ \norm{a-y}_2\colon y\in Y\}$ and $\delta:=\Big(\frac{\varepsilon}{8C_{\delta_0}}\Big)^{2^k},$ there exists  $y_0\in Y$  such that
\begin{equation}\label{E:distance}
\norm{a-y_0}_2^2\leq d^2+\delta^2.
\end{equation}
We can also assume that $\norm{y_0}_\infty\leq 1.$ Using \eqref{E:distance}  we can prove that
 \begin{equation}\label{E:ay0}
\sup_{y \in Y\colon \norm{y}_2\leq 1} |\langle a-y_0,y\rangle| \leq 2\delta.
 \end{equation}

 Since the set $\{y\in Y\colon \norm{y}_2\leq 1\}$ contains all the $(k-1)$-step nilsequences that are bounded by $1$, we have
that
\begin{equation}\label{E:uniform}
\norm{a-y_0}_{U_{k}(\N)}\leq (2\delta)^{2^{-k}}.
\end{equation}

We let
$$
\mathcal{N}:=y_0, \quad e:=a-y_0.
$$
Then
$$
a=\mathcal{N}+e,
$$
and $(\mathcal{N}(n))$ is an $(k-1)$-step nilsequence with $\norm{\mathcal{N}}_\infty\leq 1.$
Using the fact that $a$ is $k$-weak-anti-uniform and the Relation \eqref{E:uniform}, we get  that
  $$
  |\langle a,e\rangle|\leq C_{\delta_0}\norm{e}_{U_{k}(\N)}+\frac{\varepsilon}{4}=C_{\delta_0}\norm{a-y_0}_{U_{k}(\N)}+\frac{\varepsilon}{4}\leq \varepsilon/2.
  $$
Furthermore,  from  \eqref{E:ay0} we have
  $$
  |\langle \mathcal{N},e\rangle|\leq \varepsilon/2.
  $$
 Combining the last  two estimates  we deduce that
$$
 \norm{e}_2^2 =\langle e,e\rangle\leq
 |\langle a,e\rangle|+ |\langle \mathcal{N},e\rangle|
\leq \varepsilon,
$$
from which we have the result.
\end{proof}

 So, in order to prove Theorem~\ref{T:1*}, it suffices to show that the sequences of the form \eqref{E:integer} are $k$-weak-anti-uniform and $k$-regular for some $k.$  We will deal with these issues in the next two sections.

\medskip

 In order to show the $k$-regularity, we will modify a trick of M.~Wierdl (\cite{Mate}) by passing a convergence result for flows to the convergence of the respective integer parts (Theorem~\ref{T:M}). By making use of a convergence result for sequences with iterates  of integer valued polynomials due to Walsh (\cite{W12}), we will get the desired property.

\medskip

 In order to show $k$-weak-anti-uniformity, we will use, like in the previous step, an analogous trick as Wierdl's (Theorem~\ref{T:M2}) and the fact that the sequences with iterates of integer valued polynomials are $k$-anti-uniform (which we will get from \cite{F}). This last result is obtained by an inductive procedure known as PET (Polynomial Exhaustion Technique) induction, introduced in \cite{Be87a}. 

\medskip

\subsection{The seminorms $\nnorm{\cdot}_k$}

In this subsection, we follow \cite{HK99} and \cite{CFH} for the definition of the seminorms $\nnorm{\cdot}_k$ which we will use in order to prove Theorem~\ref{T:F1} and Proposition~\ref{P:F3}, which will imply Theorem~\ref{T:F2} and Corollary~\ref{C:F10}. The inductive definition that we use here follows from \cite{HK99} (in the ergodic case) and \cite{CFH} (in the general case) and the use of von Neumann's ergodic theorem.

\medskip

Let $(X,\mathcal{X},\mu,T)$ be a system and $f\in L^\infty(\mu).$  we define inductively the seminorms $\nnorm{f}_{k,\mu,T}$ as follows: 
$$ \nnorm{f}_{1,\mu,T}:= \norm{\E(f|\mathcal{I})}_{L^2(\mu)},$$
where $\mathcal{I}$ is the $\sigma$-algebra of $T$-invariant sets and $\E(f|\mathcal{I})$ is the conditional expectation of $f$ with respect to $\mathcal{I},$ satisfying $\int \E(f|\mathcal{I})\;d\mu=\int f\;d\mu$ and $ T\E(f|\mathcal{I})=\E(Tf|\mathcal{I}).$

\medskip

For $k\geq 1,$ we let
$$\nnorm{f}^{2^{k+1}}_{k+1,\mu,T}:=\lim_{N-M\to\infty}\frac{1}{N-M}\sum_{n=M}^{N-1}\nnorm{\bar{f}\cdot T^n f}^{2^k}_{k,\mu,T}. $$
All these limits exist and define seminorms (see \cite{HK99}). By using the von Neumann's ergodic theorem, we get $\nnorm{f}^2_{1,\mu,T}=\lim_{N-M\to\infty}\frac{1}{N-M}\sum_{n=M}^{N-1}\int \bar{f}\cdot T^n f\;d\mu,$ and more generally, for every $k\geq 1,$ we have that
\begin{equation}\label{E:norms}
\nnorm{f}^{2^k}_{k,\mu,T}:=\lim_{N-M\to\infty}\frac{1}{N-M}\sum_{n_1=M}^{N-1} \ldots \lim_{N-M\to\infty}\frac{1}{N-M}\sum_{n_k=M}^{N-1}\int \prod_{\vec{\epsilon}\in \{0,1\}^k}\mathcal{C}^{|\vec{\epsilon}|}T^{\vec{\epsilon}\cdot\vec{n}}f\;d\mu,
\end{equation}
 where $\vec{\epsilon}=(\epsilon_1,\ldots,\epsilon_k),$ $\vec{n}=(n_1,\ldots,n_k),$ $|\vec{\epsilon}|=\epsilon_1+\ldots+\epsilon_k,$ $\vec{\epsilon}\cdot\vec{n}=\epsilon_1 n_1+\ldots+\epsilon_k n_k$ and for $z\in \C,$ $k\in\N\cup\{0\}$ we let $ \mathcal{C}^k(z)=\left\{ \begin{array}{ll} z& \quad ,\;\text{if}\;k\;\text{is even} \\ \bar{z}&  \quad ,\;\text{if}\;k\;\text{is odd}\end{array} \right.$.

\medskip

We remark that $\nnorm{f\otimes\bar{f}}_{k,\mu\times\mu,T\times T}\leq \nnorm{f}^2_{k+1,\mu,T},$ for all $k\in \N,$ which follows from \eqref{E:norms} and the ergodic theorem, 
 $\nnorm{f}_{k,\mu,T}\leq \nnorm{f}_{k+1,\mu,T}$ for all $k\in\N$ (by using Lemma~3.9 from \cite{HK99})  and $\nnorm{f}_{k,\mu,T}=\nnorm{f}_{k,\mu,T^{-1}}$ for all $k\in\N,$ which follows from \eqref{E:norms}.

\medskip

 It is a deep fact, shown in \cite{HK99}, that for ergodic systems we have that $\nnorm{f}_{k+1}=0$ if and only if the function $f$ is orthogonal to the largest $k$-step "nil-factor" of the system. We will not use this fact here though.


\medskip

In order to recall a convergence result from \cite{CFH}, we also have to recall the notion of a nice family of polynomials.

\begin{definition*}[\cite{CFH}]
Let $\ell, m\in \N.$ Given $\ell$ families of polynomials in $\R[t]$
$$\mathcal{P}_1=(p_{1,1},\ldots,p_{1,m}),\ldots,\mathcal{P}_\ell=(p_{\ell,1},\ldots,p_{\ell,m}) $$ we define an {\em ordered family of $m$ polynomial $\ell$-tuples} as follows
$$ (\mathcal{P}_1,\ldots,\mathcal{P}_\ell)=\left((p_{1,1},\ldots,p_{\ell,1}),\ldots,(p_{1,m},\ldots,p_{\ell,m})\right).$$

In the special case where the polynomials $p_{i,j}$ belong to $\Z[t]$ we call the ordered family of polynomial $\ell$-tuples $(\mathcal{P}_1,\ldots,\mathcal{P}_\ell)$ {\em nice} if
\begin{enumerate}
\item $\deg(p_{1,1})\geq \deg(p_{1,j})$ for $1\leq j\leq m$;
\item $\deg(p_{1,1})>\deg(p_{i,j})$ for $2\leq i\leq \ell,$ $1\leq j\leq m$; and
\item $\deg(p_{1,1}-p_{1,j})>\deg(p_{i,1}-p_{i,j})$ for $2\leq i\leq \ell,$ $2\leq j\leq m.$
\end{enumerate}
A nice family of polynomials with maximum degree $1$ has only one non-zero term.
\end{definition*}

Using the theory of characteristic factors, Chu, Frantzikinakis and Host showed the following results: 

\begin{theorem}[\mbox{\cite[Theorem 1.2]{CFH}}]\label{T:F7}
For $\ell\in \N$ let $(X,\mathcal{X},\mu,T_1,\ldots,T_\ell)$ be a system, $p_1,\ldots,p_\ell$ $\in \Z[t]$ be non-constant polynomials with distinct degrees and maximum degree $d$ and functions $f_1,\ldots,f_\ell$ $\in L^\infty(\mu).$ Then there exists $k=k(\ell,d)\in \N$ such that if $\nnorm{f_i}_{k,\mu,T_i}=0$ for some $1\leq i\leq \ell,$ then the averages 
$$\frac{1}{N-M}\sum_{n=M}^{N-1}T^{p_1(n)}_1 f_1\cdot\ldots\cdot T^{p_\ell(n)}_{\ell} f_\ell $$
converge to $0$ in $L^2(\mu)$ as $N-M\to\infty.$
\end{theorem}

We will use this theorem together with Theorem~\ref{T:F4} (see below) in order to prove Theorem~\ref{T:F1}. For the proof of Theorem~\ref{T:F2}, we will use again Theorem~\ref{T:F4} and the analogous (in Proposition~\ref{P:F3}) of the following result:

\begin{proposition}[\mbox{\cite[Proposition 5.1]{CFH}}]\label{P:F8}
For $\ell\in \N$ let $(X,\mathcal{X},\mu,T_1,\ldots,T_\ell)$ be a system, $(\mathcal{P}_1,\ldots,\mathcal{P}_\ell)$ be a nice family of $\ell$-tuples of polynomials in $\Z[t]$ with maximum degree $d$ and $f_1,\ldots,f_m\in L^\infty(\mu).$ Then there exists $k=k(d,\ell,m)\in\N$ such that if $\nnorm{f_1}_{k,\mu,T_1}=0,$ then the averages 
$$\frac{1}{N-M}\sum_{n=M}^{N-1}(\prod_{i=1}^\ell T^{p_{i,1}(n)}_i)f_1\cdot\ldots\cdot (\prod_{i=1}^\ell T^{p_{i,m}(n)}_i)f_m $$
converge to $0$ in $L^2(\mu)$ as $N-M\to\infty.$
\end{proposition} 

\begin{remark*}
In these two results (Theorem~\ref{T:F7} and Proposition~\ref{P:F8}) $k$ can be chosen arbitrarily large.
\end{remark*}

\section{Regularity}

 Let $k\in \N$. In this section, we will prove that a sequence of the form \eqref{E:integer} is $k$-regular. In order to do so, we will make use of a trick due to Wierdl (Theorem~\ref{T:M} below),
a mean convergence result for multiple averages due to Walsh  (\cite{W12}) and the following proposition which we borrow from \cite{F}:
 
\begin{proposition}[\cite{F}]\label{P:nilkey}
For $k\in \N$ let  $(\psi(n))$ be a $(k-1)$-step nilsequence$.$
Then for every $\varepsilon>0$ there exists a
system $(X,\mathcal{X},\nu,S)$ and functions $f_1,\ldots, f_{k}\in L^\infty(\nu)$,
 such  that the
sequence $(b(n))$, defined by
\begin{equation}\label{E:bn}
b(n):=\int  S^{\ell_1n}f_1 \cdot \ldots \cdot  S^{\ell_{k} n}f_{k}\ d\nu, \quad n\in \N,
\end{equation}
where $\ell_i:=k!/i$ for $i=1,\ldots, k$, satisfies
$$
\norm{\psi-b}_\infty\leq \varepsilon.
$$
\end{proposition}

In order to obtain the $k$-regularity for any $k\in\N$ of a sequence  $(a(n))$ as in  \eqref{E:integer}, we have to check that  the limit
 $\displaystyle \lim_{N-M\to \infty} \frac{1}{N-M}\sum_{n=M}^{N-1}a(n)\psi(n) $ exists for every $(k-1)$-step nilsequence $(\psi(n))$. From  Proposition~\ref{P:nilkey}, it suffices to check that the limit
  \begin{equation}\label{E:abnn}
\lim_{N-M\to \infty} \frac{1}{N-M}\sum_{n=M}^{N-1} a(n)b(n)
\end{equation}
 exists for every sequence
  $(b(n))$ of the form $\int  S^{\ell_1n}g_1\cdot \ldots \cdot S^{\ell_{k} n}g_{k}\ d\nu$ , where $\ell_1,\ldots, \ell_k\in \N$,  $(Y,\mathcal{Y},\nu, S)$ is a system, and $g_1,\ldots, g_{k}\in L^\infty(\nu)$. 

\medskip

To verify that the limit in \eqref{E:abnn} exists, we will use a trick of M.~Wierdl (the result of Wierdl is for $\R^2$ measure preserving flows, see definition below, and Cesaro averages. We will translate his proof in our setting).

\medskip

 We first recall the notion of the  upper Banach density.
 
\begin{definition*} 
 Let $S$ be a subset of natural numbers. We define the {\em upper Banach density} of $S,$ $d^{\ast}(S),$ to be the number $$d^{\ast}(S)=\limsup_{N-M\to\infty}\frac{|S\cap [M,\ldots, N)|}{N-M}.$$
\end{definition*} 

Also, we recall the notion of a measure preserving flow.

\begin{definition*}
Let $r\in \mathbb{N}$ and $(X,\mathcal{X}, \mu)$ be a probability space. 
We call a family $(T_t)_{t\in \mathbb{R}^r}$ of measure preserving 
transformations $T_t\colon X\to X,$ a {\em measure preserving flow}, if 
it satisfies $$T_{s+t}=T_s\circ T_t$$ for all $s,t\in \R^r.$
\end{definition*}

The following theorem contains the central idea for passing from results for flows to results for $\Z$-actions. Essential for this, is the $(\ell m)$-dimensional variant of the special flow above a system under the constant ceiling function $1$ (see in the proof below) first defined (for $\ell=m=1$) in \cite{EL}.

\begin{theorem}[\cite{Mate}]\label{T:M}
Let $\ell, m \in \N.$  Suppose that the sequences of real numbers $(a_{i,j}(n))$ for $1\leq i\leq \ell,$ $1\leq j\leq m,$ satisfy the following two properties:

\begin{enumerate}
\item For any $\R^{\ell m}$ measure preserving flow $\displaystyle \prod_{i=1}^{\ell}T_{i,a_{i,1}}\cdot \ldots\cdot \prod_{i=1}^\ell T_{i,a_{i,m}},$ where the transformations $T_{i,a_{i,j}}$ are defined in the probability space $(X,\mathcal{X},\mu)$ and functions $f_1,\ldots, f_m$ $\in L^\infty(\mu),$ the averages $$\frac{1}{|I_N|}\sum_{n\in I_N}  (\prod_{i=1}^\ell T_{i,a_{i,1}(n)})f_1\cdot \ldots\cdot (\prod_{i=1}^\ell T_{i,a_{i,m}(n)})f_m$$ converge in $L^2(\mu)$ as $N\to\infty$ (where $|I_N|\to\infty$ as $N\to\infty$);

\item  $\displaystyle \lim_{\delta\to 0^+} d^{\ast}\left(\Big\{n\colon \{a_{i,j}(n)\}\in [1-\delta,1)\Big\}\right)=0,$ for all $1\leq i\leq \ell,$  $1\leq j\leq m,$ where $\{\cdot\}$ denotes the fractional part$.$
\end{enumerate}

Then the averages $$\frac{1}{|I_N|}\sum_{n\in I_N} (\prod_{i=1}^\ell T_i^{[a_{i,1}(n)]})f_1\cdot\ldots\cdot(\prod_{i=1}^\ell T_i^{[a_{i,m}(n)]})f_m$$ also converge in $L^2(\mu),$ as $N\to\infty,$ for every system $(X,\mathcal{X},\mu,T_1,\ldots, T_\ell)$ and functions 
$f_1,\ldots, f_m\in L^\infty(\mu).$
\end{theorem}

\begin{proof}[Proof (We use the notation and arguments of Wierdl from \cite{Mate})]
For the given transformations on $X,$ we define the $\R^{\ell m}$ action $\displaystyle \prod_{i=1}^{\ell}T_{i,a_{i,1}}\cdot \ldots\cdot \prod_{i=1}^\ell T_{i,a_{i,m}}$  on the probability space $Y=X\times [0,1)^{\ell m},$ with the measure $\nu=\mu\times \lambda^{\ell m}$ ($\lambda$ is the Lebesgue measure on $[0,1)$), by
$$\prod_{j=1}^m \prod_{i=1}^\ell T_{i,a_{i,j}}(x,b_{1,1},\ldots,b_{\ell,1},b_{1,2},\ldots,b_{\ell,2},\ldots,b_{1,m},\ldots,b_{\ell,m})=$$
$$\left(\prod_{j=1}^m \prod_{i=1}^\ell T_i^{[a_{i,j}+b_{i,j}]}x,\{a_{1,1}+b_{1,1}\},\ldots,\{a_{\ell,1}+b_{\ell,1}\},\ldots,\{a_{1,m}+b_{1,m}\},\ldots,\{a_{\ell,m}+b_{\ell,m}\}\right). $$
Since the transformations $T_1,\ldots, T_\ell$ are measure preserving and commute, and since we have $[x+\{y\}]+[y]=[x+y],$ it is easy to check that the above action define a measure preserving flow on the product probability space $Y$.

 Note that this is nothing else than the ($\ell m$)-dimensional variant of the special flow above a system under the constant ceiling function $1.$

For a bounded function $f$ on $X,$ we define its version $\hat{f}$ on $Y$ by
$$\hat{f}(x,b_{1,1},\ldots,b_{\ell,1},b_{1,2},\ldots,b_{\ell,2},\ldots,b_{1,m},\ldots,b_{\ell,m})=f(x).$$
Note that if  $a_{i,j},$ for $1\leq i\leq \ell,$ $1\leq j\leq m,$ are real numbers, then
$$ \prod_{j=1}^m (\prod_{i=1}^\ell T_{i,a_{i,j}})\hat{f}_j(x,0,\ldots,0)=\prod_{j=1}^m (\prod_{i=1}^\ell T_i^{[a_{i,j}]})f_j(x).$$
We want to show that the averages $$\frac{1}{|I_N|}\sum_{n\in I_N}\prod_{j=1}^m (\prod_{i=1}^\ell T_{i,a_{i,j}(n)})\hat{f}_j(x,0,\ldots,0)$$ converge in $L^2(X,0,\ldots,0)$ as $N\to\infty.$ Denote 
$$A_N(x,b_{1,1},\ldots,b_{\ell,m})=\frac{1}{|I_N|}\sum_{n\in I_N}\prod_{j=1}^m (\prod_{i=1}^\ell T_{i,a_{i,j}(n)})\hat{f}_j(x,b_{1,1},\ldots,b_{\ell,m})$$ and assume to the contrary that $(A_N(x,0,\ldots,0))$ is not Cauchy in $L^2(X,0,\ldots,0).$ This means that we can find a sequence $(N_k)$ going to infinity, with 
\begin{equation}\label{E:c}
\int_{(X,0,\ldots,0)}\Big|A_{N_{k+1}}(x,0,\ldots,0)-A_{N_k}(x,0,\ldots,0)\Big|^2\,d\mu >c,\;\;k=1,2,\ldots,
\end{equation}
for some positive number $c.$ By the hypothesis, $(A_N(x,b_{1,1},\ldots,b_{\ell,m}))$ is Cauchy in $L^2(Y).$ By Fubini's theorem, for any given positive $\varepsilon,$ we can find $(b_{1,1},\ldots,b_{\ell,m})$ arbitrarily close to $(0,\ldots,0)$ and $k\in \N$ so that 
$$
\int_{(X,b_{1,1},\ldots,b_{\ell,m})}\Big|A_{N_{k+1}}(x,b_{1,1},\ldots,b_{\ell,m})-A_{N_k}(x,b_{1,1},\ldots,b_{\ell,m})\Big|^2\, d\mu<\varepsilon.
$$
We will show that for any given $\varepsilon>0,$ there exists $\delta>0$ so that if $0\leq b_{i,j}\leq\delta,$ for $1\leq i\leq \ell,$ $1\leq j\leq m,$ then, for every $x\in X$ and large enough  $N,$ we have
\begin{equation}\label{E:cc} \Big|A_{N}(x,b_{1,1},\ldots,b_{\ell,m})-A_{N}(x,0,\ldots,0)\Big|<\varepsilon,
\end{equation}
 and so, we will obtain the required contradiction with the Relation \eqref{E:c}.

Let $0<\delta<1$ (we will choose it later), and assume that $0\leq b_{i,j}\leq\delta,$ for all $1\leq i\leq \ell,$ $1\leq j\leq m.$ In \eqref{E:cc} we have to compare terms of the form 
$$\prod_{j=1}^m \hat{f}_j\left(\prod_{i=1}^\ell T_{i,a_{i,j}(n)}(x,b_{1,1},\ldots,b_{\ell,m})\right) \;\;\;\text{with}\;\;\;\prod_{j=1}^m \hat{f}_j\left(\prod_{i=1}^\ell T_{i,a_{i,j}(n)}(x,0,\ldots,0)\right).$$
So, by the definition of the flow, we need to compare terms of the form 
$$\hat{f}_1\left(\prod_{i=1}^\ell T_i^{[a_{i,1}(n)+b_{i,1}]}x,\{a_{1,1}(n)+b_{1,1}\},\ldots, \{a_{\ell,1}(n)+b_{\ell,1}\},b_{1,2},\ldots,b_{\ell,m} \right)\cdot \ldots$$ $$\ldots\cdot \hat{f}_m\left(\prod_{i=1}^\ell T_i^{[a_{i,m}(n)+b_{i,m}]}x,b_{1,1},\ldots,b_{\ell,m-1},\{a_{1,m}(n)+b_{1,m}\},\ldots,\{a_{\ell,m}(n)+b_{\ell,m}\} \right) $$
with
$$\hat{f}_1\left(\prod_{i=1}^\ell T_i^{[a_{i,1}(n)]}x,\{a_{1,1}(n)\},\ldots, \{a_{\ell,1}(n)\},0,\ldots,0 \right)\cdot \ldots$$ $$\ldots\cdot \hat{f}_m\left(\prod_{i=1}^\ell T_i^{[a_{i,m}(n)]}x,0,\ldots,0,\{a_{1,m}(n)\},\ldots,\{a_{\ell,m}(n)\} \right),$$ or equivalently, by the definition of $\hat{f}_j$, we need to compare $$\prod_{j=1}^m f_j(\prod_{i=1}^\ell T_i^{[a_{i,j}(n)+b_{i,j}]}x)\;\;\;\text{with}\;\;\;\prod_{j=1}^m f_j(\prod_{i=1}^\ell T_i^{[a_{i,j}(n)]}x).$$

Since all $b_{i,j}$ are less or equal than $\delta,$ if the fractional part of all $a_{i,j}(n)$ is less than $1-\delta,$ we have that $T_i^{[a_{i,j}(n)+b_{i,j}]}=T_i^{[a_{i,j}(n)]}$ for all $1\leq i\leq \ell,$ $1\leq j\leq m.$

It remains to deal with those $n$'s for which the fractional part of some $a_{i,j}(n)$ is greater or equal than $1-\delta.$ By the Condition (ii), we have that the upper Banach density of these n's is as small as we want by taking $\delta$ small. All $f_j$ are bounded, and so, the combined density of these terms in the averages $A_N(x,b_{1,1},\ldots,b_{\ell,m})$ and $A_N(x,0,\ldots,0)$ will be as small as we want independently of $x.$ Hence,  if $\delta$ is chosen sufficiently small, we get \eqref{E:cc}.
\end{proof} 

Now, by using the previous result, we will prove Theorem~\ref{T:3*}.
   
\begin{proof}[Proof of Theorem~\ref{T:3*}]
 It  suffices to show that for $a_{i,j}=p_{i,j},$ real valued polynomials, we have the conditions of Theorem~\ref{T:M}.
 
 Using Walsh's convergence result for commuting measure preserving transformations from \cite{W12} we have the Condition (i). Indeed, if, for example, $p(t)=a_r t^r+\ldots+a_1 t+a_0\in \R[t],$ we write $T_{p(n)}=(T_{a_r})^{n^r}\cdot \ldots\cdot (T_{a_1})^{n}\cdot T_{a_0}$ and we use Walsh's result for the commuting measure 
preserving transformations $S_1=T_{a_1},\ldots, S_r=T_{a_r}.$
 
  Real valued polynomials also satisfy the Condition (ii) of the previous theorem. 
  
  \noindent Indeed, let $p(t)=a_r t^r+\ldots+a_1 t+a_0\in \R[t].$

If $a_i\notin \Q$ for some $1\leq i\leq r,$ then we have the condition from Weyl's result, since $(p(n))$ is uniformly distributed (mod 1).

If $a_i\in \mathbb{Q}$ for all $1\leq i\leq r,$ then the sequence $(p(n))$ is periodic (mod 1) and 
Condition~(ii) is obvious. 
\end{proof}   
 
In order to show the $k$-regularity, it is sufficient to  show that the limit \eqref{E:abnn} exists for every sequence
  $(b(n))$ of the form $\int  S^{\ell_1n}g_1\cdot \ldots \cdot S^{\ell_r n}g_r\ d\nu$,
   where $r\in \N$ is arbitrary, $\ell_1,\ldots, \ell_r\in \N$, $(Y,\mathcal{Y},\nu, S)$ is a system, and $g_1,\ldots, g_r\in L^\infty(\nu).$ 
   
   \medskip
   
We want to show that the averages of $$\int f_0\cdot(\prod_{i=1}^\ell T_i^{[p_{i,1}(n)]})f_1\cdot \ldots\cdot (\prod_{i=1}^\ell T_i^{[p_{i,m}(n)]})f_m\, d\mu\cdot \int  S^{\ell_1n}g_1\cdot \ldots \cdot S^{\ell_r n}g_r\ d\nu$$ converge, so, we want to show that the averages of $$ (\prod_{i=1}^\ell T_i^{[p_{i,1}(n)]})f_1\cdot \ldots\cdot (\prod_{i=1}^\ell T_i^{[p_{i,m}(n)]})f_m \cdot \int  S^{\ell_1n}g_1\cdot \ldots \cdot S^{\ell_r n}g_r\ d\nu$$ converge in $L^2(\mu).$ 
   
\medskip   
   
 We will use the convergence consequence from Theorem~\ref{T:3*} for
the $\ell+r$ commuting measure preserving  transformations
$T_i\times \text{id},$ $i=1,\ldots, \ell,$ and  $\text{id}\times S^{\ell_j},$ $j=1,\ldots, r,$  acting on $X\times Y$ with the measure $\tilde{\mu}:=\mu\times \nu,$ and the functions $f_i\otimes 1,$ $i=1,\ldots, \ell$  and $1\otimes g_j,$ $j=1, \ldots, r$. 

\medskip

By Theorem~\ref{T:3*} we have that the averages of  $$ \prod_{j=1}^m\left((\prod_{i=1}^\ell (T_i\times \text{id})^{[p_{i,j}(n)]})(f_j\otimes 1)\right) \cdot \prod_{j=1}^r\left( (\text{id}\times S^{\ell_j})^n(1\otimes g_j)\right)$$ converge in $L^2(\mu\times \nu),$ and so, integrating by $d\nu,$ we get the required convergence.

\section{Weak-Anti-Uniformity} 

 In this section we will show that a sequence of the form \eqref{E:integer} is $k$-weak-anti-uniform, for some $k$ depending only on $\ell,$ $m$ and the maximum degree of  the polynomials $p_{i,j}.$
 In order to do so, we need a result (see Theorem~\ref{T:M2} below), that will allow us to pass from known results for flows, to $\Z$ actions.

\medskip

  As we saw in the previous section, in order to get the required convergence result to show $k$-regularity for sequences of the form \eqref{E:integer}, we used an argument (Theorem~\ref{T:M}) and a known, from \cite{W12}, convergence result for flows. We will now prove a similar result that will allow us to do the analogous thing in order to obtain the weak-anti-uniformity. The known result for (some particular) flows in this case will be the anti-uniformity that we obtain from \cite{F}.
 More specifically, we will show that the $k$-anti-uniformity of (some particular) flows will give us the $k$-weak-anti-uniformity for the sequences of the form \eqref{E:integer} (for the same $k$).

\begin{theorem}\label{T:M2}
Let $\ell, m\in \N.$ Suppose that the sequences of real numbers $(a_{i,j}(n))$  satisfy the following two properties:

\begin{enumerate}
\item For any $\R^{\ell m}$ measure preserving flow $\displaystyle \prod_{i=1}^{\ell}T_{i,a_{i,1}}\cdot \ldots\cdot \prod_{i=1}^\ell T_{i,a_{i,m}},$ where the transformations $T_{i,a_{i,j}}$ are defined in the probability space $(X,\mathcal{X},\mu)$ and functions $f_0, f_1,\ldots,$ $ f_m\in L^\infty(\mu),$ the sequence $$\tilde{a}(n)=\int f_0\cdot(\prod_{i=1}^\ell T_{i,a_{i,1}(n)})f_1\cdot \ldots\cdot (\prod_{i=1}^\ell T_{i,a_{i,m}(n)})f_m\,d\mu$$ is $k$-anti-uniform, for some $k$ depending only on $\ell,$ $m$ and $a_{i,j};$

\item  $\displaystyle \lim_{\delta\to 0^+} d^{\ast}\left(\Big\{n\colon \{a_{i,j}(n)\}\in [1-\delta,1)\Big\}\right)=0,$ for all $1\leq i\leq \ell,$  $1\leq j\leq m,$ where $\{\cdot\}$ denotes the fractional part.
\end{enumerate}

\noindent Then, the sequence $$a(n)=\int f_0\cdot(\prod_{i=1}^\ell T_i^{[a_{i,1}(n)]})f_1\cdot\ldots\cdot(\prod_{i=1}^\ell T_i^{[a_{i,m}(n)]})f_m\,d\mu$$ is $k$-weak-anti-uniform  for every system $(X,\mathcal{X},\mu,T_1,\ldots, T_\ell)$ and functions 
$f_0,f_1,\ldots, f_m\in L^\infty(\mu).$
\end{theorem}

\begin{proof}
Let $0<\delta<1.$ We define  the same action $\R^{\ell m}$ on $Y=X\times [0,1)^{\ell m}$ as we did in the proof of Theorem~\ref{T:M}. If $f_0,f_1,\ldots,f_m$ are bounded functions on $X,$ for every $(b_{1,1},\ldots,b_{\ell,1},b_{1,2},\ldots,b_{\ell,2},\ldots,b_{1,m},\ldots,b_{\ell,m})\in [0,1)^{\ell m}$   we define the $Y$-extensions $$\hat{f}_j(x,b_{1,1},\ldots,b_{\ell,1},b_{1,2},\ldots,b_{\ell,2},\ldots,b_{1,m},\ldots,b_{\ell,m})=f_j(x),\;\;1\leq j\leq m;$$  and $$\hat{f}_0(x,b_{1,1},\ldots,b_{\ell,1},b_{1,2},\ldots,b_{\ell,m})=f_0(x)\cdot {\bf 1}_{[0,\delta]^{\ell m}}(b_{1,1},\ldots,b_{\ell,1},b_{1,2},\ldots,b_{\ell,m}).$$ Then$,$  we have:
$$\Big|\delta^{\ell m} a(n)-\tilde{a}(n)\Big|= $$
$$\Big|\int_{[0,\delta]^{\ell m}} \int_X f_0(x)\cdot \left(\prod_{j=1}^m f_j(\prod_{i=1}^\ell T_i^{[a_{i,j}(n)]}x)-\prod_{j=1}^m f_j(\prod_{i=1}^\ell T_i^{[a_{i,j}(n)+b_{i,j}]}x) \right)\,d\mu  \, d \lambda^{\ell m}\Big|. $$
 Since all the relevant $b_{i,j}$ in the integrand are less or equal than $\delta,$ if the fractional part of all $a_{i,j}(n)$ is less than $1-\delta,$ we have $T_i^{[a_{i,j}(n)+b_{i,j}]}=T_i^{[a_{i,j}(n)]}$ for all $1\leq i\leq \ell,$ $1\leq j\leq m.$ If the fractional part of some $a_{i,j}(n)$ is greater or equal than $1-\delta,$ by the Condition~(ii), we have that the upper Banach density of these $n$'s is as small as we want by taking $\delta$ small. All $f_j$ are bounded, and so, the combined density of these terms in the averages will be as small as we want independently of $x.$

Hence, by taking the averages and using the triangle inequality$,$ for any $(b(n))\in \ell^\infty(\N)$ and $N>M,$ we have that $$\Big|\frac{1}{N-M} \sum_{n=M}^{N-1}a(n)b(n)\Big|\leq \frac{1}{\delta^{\ell m}}\Big|\frac{1}{N-M} \sum_{n=M}^{N-1}\tilde{a}(n)b(n)\Big| +c_\delta, $$ where $c_\delta \to 0$ as $\delta\to 0^+.$

Then, if $C$ is the constant that we get from the $k$-anti-uniformity of $\tilde{a},$ we have
$$\limsup_{N-M\to\infty}\Big|\frac{1}{N-M} \sum_{n=M}^{N-1}a(n)b(n)\Big|\leq \frac{C}{\delta^{\ell m}}\cdot \norm{b}_{U_k(\N)}+c_\delta, $$
 and so$,$ we have the result.
\end{proof}

\begin{remark*}\label{R:2}
As we showed in the proof of Theorem~\ref{T:3*}$,$ if every $a_{i,j}$ is a polynomial $p_{i,j} \in \R[t],$ then the Condition (ii) of Theorem~\ref{T:M2} is satisfied. In order to get the Condition (i), as described by Theorem 1.2 in \cite{F}, for the corresponding sequences, we have to successively make use of Lemma~\ref{L:VDC}  (using the van der Corput operation, choosing every time appropriate polynomials in order to have reduction in our complexity), defined in Section 6 below. $k$ can be chosen to be equal to $d+1,$ where $d$ is the number of steps we need to do in order our polynomials to be reduced into constant ones, by using the PET induction. This $d,$ and so $k$ as well, only depends on $\ell,$ $m$ and the maximum degree of the polynomials $p_{i,j}.$

 For more information and details on the van der Corput operation and the scheme of the PET induction we are using here, we refer the reader to \cite{FrHK11}.
\end{remark*}

 So, by using the previous remark and Theorem~\ref{T:M2}, we have that every sequence $(a(n))$ of the form \eqref{E:integer} is $k$-weak-anti-uniform, for some positive integer $k=k(\ell,m,\max \deg(p_{i,j})).$
 
\section{Convergence}

In this section we will present all the ingredients in order to prove  Theorems~\ref{T:F1} (which we prove in Section 6) and ~\ref{T:F2} (which we prove below). More specifically, in order to derive these results, we prove Theorem~\ref{T:F4} below, which is yet another result for proving results for $\Z$-actions  via  known results for flows. In particular, Theorem~\ref{T:F1} will follow from Theorem~\ref{T:F4} and the analogous result from \cite{CFH} (Theorem~1.2) while Theorem~\ref{T:F2}, which is the analogous result to Corollary~5.2 from \cite{CFH}, is an implication of Proposition~\ref{P:F3} which follows from Theorem~\ref{T:F4} and a result  from \cite{CFH} (\mbox{\cite[Proposition 5.1]{CFH}}).

\medskip

For $\ell, m\in \N$ and a system $(X,\mathcal{X},\mu,T_1,\ldots, T_\ell),$ recall the definition of the $\R^{\ell m}$ measure preserving flow $\displaystyle \prod_{i=1}^{\ell}T_{i,a_{i,1}}\cdot \ldots\cdot \prod_{i=1}^\ell T_{i,a_{i,m}}$ on the space $Y=X\times [0,1)^{\ell m}$ that we defined in the proof of Theorem~\ref{T:M}. Also, If $f_1,\ldots,f_m\in L^2(\mu),$ for every $(b_{1,1},\ldots,b_{\ell,1},b_{1,2},\ldots,b_{\ell,m})\in [0,1)^{\ell m}$ let $\hat{f}_j(x,b_{1,1},\ldots,b_{\ell,1},b_{1,2},\ldots,b_{\ell,m})=f_j(x),$ $1\leq j\leq m$ be the $Y$-extensions of $f_j.$

\medskip

Following this notation, we have:

\begin{theorem}\label{T:F4}
Let $\ell, m\in \N,$ $(X,\mathcal{X},\mu,T_1,\ldots, T_\ell)$ a system, $f_1,\ldots, f_m\in L^\infty(\mu)$ and sequences of real numbers $(a_{i,j}(n)),$ $1\leq i\leq \ell,$ $1\leq j\leq m,$  that satisfy the following:
\begin{enumerate}
\item For the $\R^{\ell m}$ action $\displaystyle \prod_{i=1}^{\ell}T_{i,a_{i,1}}\cdot \ldots\cdot \prod_{i=1}^\ell T_{i,a_{i,m}}$ on the space $Y,$ endowed with the probability measure $\nu=\mu\times\lambda^{\ell m},$ and the extensions $\hat{f}_1,\ldots,$ $ \hat{f}_m\in L^\infty(\nu),$ we have 
$$\lim_{N-M\to\infty}\norm{\frac{1}{N-M}\sum_{n=M}^{N-1}(\prod_{i=1}^\ell T_{i,a_{i,1}(n)})\hat{f}_1\cdot \ldots\cdot (\prod_{i=1}^\ell T_{i,a_{i,m}(n)})\hat{f}_m}_{L^2(\nu)}=0 ;$$

\item  $\displaystyle \lim_{\delta\to 0^+} d^{\ast}\left(\Big\{n\colon \{a_{i,j}(n)\}\in [1-\delta,1)\Big\}\right)=0,$ for all $1\leq i\leq \ell,$  $1\leq j\leq m.$
\end{enumerate}
\noindent Then,  we have 
$$\lim_{N-M\to\infty}\norm{\frac{1}{N-M}\sum_{n=M}^{N-1} (\prod_{i=1}^\ell T_i^{[a_{i,1}(n)]})f_1\cdot\ldots\cdot(\prod_{i=1}^\ell T_i^{[a_{i,m}(n)]})f_m}_{L^2(\mu)}=0.$$
\end{theorem}

\begin{proof}
Let $0<\delta<1$ and define the function $\hat{f}_0$ in $Y$ with $$\hat{f}_0(x,b_{1,1},\ldots,b_{\ell,1},b_{1,2},\ldots,b_{\ell,m})= {\bf 1}_{[0,\delta]^{\ell m}}(b_{1,1},\ldots,b_{\ell,1},b_{1,2},\ldots,b_{\ell,m}).$$ If $\;\tilde{a}(n)= \hat{f}_0\cdot(\prod_{i=1}^\ell T_{i,a_{i,1}(n)})\hat{f}_1\cdot \ldots\cdot (\prod_{i=1}^\ell T_{i,a_{i,m}(n)})\hat{f}_m,$ 

 for every $x\in X$ we define $$a'(n)(x)=\int_{[0,1)^{\ell m}}\tilde{a}(n)(x,b_{1,1},\ldots,b_{\ell,1},b_{1,2},\ldots,b_{\ell,m})\,d\lambda^{\ell m},$$ where the integration is with respect to the variables $b_{i,j}.$

Then, if $a(n):=(\prod_{i=1}^\ell T_i^{[a_{i,1}(n)]})f_1\cdot\ldots\cdot(\prod_{i=1}^\ell T_i^{[a_{i,m}(n)]})f_m,$ for every $x\in X,$ we have that 
$$\Big|\delta^{\ell m} a(n)(x)-a'(n)(x)\Big|= $$
$$\Big|\int_{[0,\delta]^{\ell m}}  \left(\prod_{j=1}^m f_j(\prod_{i=1}^\ell T_i^{[a_{i,j}(n)]}x)-\prod_{j=1}^m f_j(\prod_{i=1}^\ell T_i^{[a_{i,j}(n)+b_{i,j}]}x) \right)\, d \lambda^{\ell m}\Big|. $$
Since all the relevant $b_{i,j}$ in the integrand are less or equal than $\delta,$ if the fractional part of all $a_{i,j}(n)$ is less than $1-\delta,$ we have $T_i^{[a_{i,j}(n)+b_{i,j}]}=T_i^{[a_{i,j}(n)]}$ for all $1\leq i\leq \ell,$ $1\leq j\leq m.$ We will deal with the case where the fractional part of some $a_{i,j}(n)$ is greater or equal than $1-\delta.$ 

For every $1\leq i\leq \ell,$ $1\leq j\leq m,$ let
$$E_{\delta}^{i,j}:=\{n\in \N\colon \{a_{i,j}(n)\}\in [1-\delta,1)\}.$$
Then, by using the fact that ${\bf 1}_{E_\delta^{1,1}\cup \ldots\cup E_\delta^{1,m}\cup E_\delta^{2,1}\cup\ldots\cup E_\delta^{\ell,m}}\leq \sum_{(i,j)\in[1,\ell]\times[1,m]} {\bf 1}_{E_{\delta}^{i,j}} $ and that ${\bf 1}_{E_{\delta}^{i,j}}(n)={\bf 1}_{[1-\delta,1)}(\{a_{i,j}(n)\}),$ for $1\leq i\leq \ell,$ $1\leq j\leq m,$ $n\in\N,$  we have
$$\frac{1}{\delta^{\ell m}}\norm{\frac{1}{N-M}\sum_{n=M}^{N-1} (\delta^{\ell m}a(n)-a'(n))}_{L^2(\mu)}\leq 2\sum_{(i,j)\in[1,\ell]\times[1,m]}\Big| \frac{1}{N-M}\sum_{n=M}^{N-1} {\bf 1}_{[1-\delta,1)}(\{a_{i,j}(n)\}) \Big|,$$ where
$$\Big| \frac{1}{N-M}\sum_{n=M}^{N-1} {\bf 1}_{[1-\delta,1)}(\{a_{i,j}(n)\}) \Big|=\frac{|E_{\delta}^{i,j}\cap [M,N)|}{N-M}.$$
Using Condition (ii), we have that for small enough $\delta,$ the term (and the sum of finitely many terms of this form)  $\frac{|E_{\delta}^{i,j}\cap [M,N)|}{N-M}$ is as small as we want. Since,
$$\delta^{\ell m}\norm{\frac{1}{N-M}\sum_{n=M}^{N-1} a(n)}_{L^2(\mu)}\leq \norm{\frac{1}{N-M}\sum_{n=M}^{N-1} (\delta^{\ell m}a(n)-a'(n))}_{L^2(\mu)}+\norm{\frac{1}{N-M}\sum_{n=M}^{N-1} \tilde{a}(n)}_{L^2(\nu)}$$ we have that $$ \norm{\frac{1}{N-M}\sum_{n=M}^{N-1} a(n)}_{L^2(\mu)}\leq c_\delta +\delta^{-\ell m}\norm{\frac{1}{N-M}\sum_{n=M}^{N-1} \tilde{a}(n)}_{L^2(\nu)},$$ where $c_\delta\to 0$ as $\delta\to 0^+.$ We first take $\limsup_{N-M\to \infty},$ in order the second term of the right-hand side to disappear from Condition (i), and then $\delta\to 0^+$ to get the result. 
\end{proof}

We will also need the following elementary estimate (for simplicity, we use the notation $\text{U-}\limsup_{n_1,\ldots,n_k}\E_{n_1,\ldots,n_k}$ instead of $\limsup_{N-M\to\infty}\frac{1}{N-M}\sum_{n_1=M}^{N-1}\ldots \limsup_{N-M\to\infty}\frac{1}{N-M}\sum_{n_k=M}^{N-1}$).

\begin{lemma}\label{L:F5}
Let $k\in\N$ and $s\in (0,+\infty).$ For any sequence $(a(n_1,\ldots,n_k))$ of real non-negative numbers we have 
$$\text{{\em U}-}\limsup_{n_1,\ldots,n_k}\E_{n_1,\ldots,n_k} a([n_1 s],\ldots, [n_k s])\leq s^k(\Big[\frac{1}{s}\Big]+1)^k \text{{\em U}-}\limsup_{n_1,\ldots,n_k}\E_{n_1,\ldots,n_k} a(n_1,\ldots, n_k).$$
\end{lemma}

\begin{proof}
For $k=1$ we have
\begin{eqnarray*}
\frac{1}{N-M}\sum_{n=M}^{N-1}a([ns]) & \leq & (\Big[\frac{1}{s}\Big]+1)\frac{1}{N-M}\sum_{n=[Ms]}^{[(N-1)s]}a(n) \\
& = & (\Big[\frac{1}{s}\Big]+1)\frac{[(N-1)s]-[Ms]}{N-M}\frac{1}{[(N-1)s]-[Ms]}\sum_{n=[Ms]}^{[(N-1)s]}a(n).
\end{eqnarray*} Since $\lim_{N-M\to\infty}\frac{[(N-1)s]-[Ms]}{N-M}=s,$ by taking $\limsup_{N-M\to\infty}$ will give us the required relation.

The general $k>1$ case follows analogously with induction.
\end{proof}

For $\ell, m\in \N,$ $(X,\mathcal{X},\mu,T_1,\ldots,T_\ell)$ system and $f_1,\ldots,f_m\in L^\infty(\mu)$, recall the action $\displaystyle \prod_{i=1}^{\ell}T_{i,a_{i,1}}\cdot \ldots\cdot \prod_{i=1}^\ell T_{i,a_{i,m}}$ that we defined in the proof of Theorems~\ref{T:M}, ~\ref{T:M2} and ~\ref{T:F4} and the $Y$-extensions, $\hat{f}_j$ of $f_j,$ where $Y=X\times [0,1)^{\ell m}$ is endowed with the probability measure $\nu=\mu\times\lambda^{\ell m}.$

\medskip

 By the definition of the action, the first coordinate of $T_{i_0,a_{i_0,j_0}}$ evaluated at the point $(x,b_{1,1},\ldots,b_{\ell,1},b_{1,2},\ldots,b_{\ell,m})\in Y$ is $T^{[a_{i_0,j_0}+b_{i_0,j_0}]}_{i_0}x,$ the $((j_{0}-1)\ell+i_0+1)$-coordinate is equal to $\{a_{i_0,j_0}+b_{i_0,j_0}\},$ while in all the other coordinates we have the identity map, mapping $b_{i,j}$ to itself.

\medskip

 So, without loss of generality, in order to study the transformations $T_{i,s},$ $s\in \R,$ which we will essentially use in the proof of Theorems~\ref{T:F1} and ~\ref{T:F2},  we restrict our study to the $\ell=m=1$ case, studying the transformation $S=T_{s},$ where $s\in (0,+\infty)$ (in the case where $s<0,$ we set $S=T^{-1}_{-s}$).  Hence, we study the transformation $S(x,b)=(T^{[s+b]}x,\{s+b\}).$ 
 
\medskip 
 
  By making use of the relations $[s+\{s'\}]+[s']=[s+s']$ and $\{s+\{s'\}\}=\{s+s'\},$ we have that $S^n(x,b)=(T^{[ns +b]}x,\{ns+b\}),$ and so, $S^n \hat{f}(x,b)=T^{[ns+b]}f(x).$ 

\medskip

The next important lemma will give us a relation between $\nnorm{\hat{f}}_{k,\nu,S}$ and $\nnorm{f}_{k,\mu,T}$ (recall the definitions and remarks from Subsection 2.2).

\begin{lemma}\label{L:F6}
With the previous terminology, for any $k\in \N,$ there exists a constant $c=c(k,s)$ such that $$\nnorm{\hat{f}}_{k,\nu,S}\leq c \nnorm{f}_{k+1,\mu,T}.$$
\end{lemma}

\begin{proof}
If $c_k=(k+1)^{2^k},$ $c_{k,s}=c_k\cdot s^k(\Big[\frac{1}{s}\Big]+1)^k,$ $F=f\otimes\bar{f}$ and $R=T\times T,$ then by the definition of the seminorm $\nnorm{\cdot}_k,$ by Lemma~\ref{L:F5}, the Cauchy-Schwarz inequality and the remarks in Subsection 2.2, we have (for simplicity, we use the notation $\text{U-}\lim_{n_1,\ldots,n_k}\E_{n_1,\ldots,n_k}$ instead of $\lim_{N-M\to\infty}\frac{1}{N-M}\sum_{n_1=M}^{N-1}\ldots \lim_{N-M\to\infty}\frac{1}{N-M}\sum_{n_k=M}^{N-1}$ and $\text{U-}\limsup_{n_1,\ldots,n_k}\E_{n_1,\ldots,n_k}$ for the $\limsup$ respectively, as we did in Lemma~\ref{L:F5})

\begin{eqnarray*}
\nnorm{\hat{f}}^{2^k}_{k,\nu,S} &=& \text{U-}\lim_{n_1,\ldots,n_k}\E_{n_1,\ldots,n_k}\int \prod_{\vec{\epsilon}\in \{0,1\}^k}\mathcal{C}^{|\epsilon|}S^{\vec{\epsilon}\cdot \vec{n}}\hat{f}\;d\nu \\ 
& \leq & \text{U-}\limsup_{n_1,\ldots,n_k}\E_{n_1,\ldots,n_k} \Big|\int \prod_{\vec{\epsilon}\in \{0,1\}^k}\mathcal{C}^{|\epsilon|}T^{[(\vec{\epsilon}\cdot \vec{n})s+b]}f\;d\nu \Big| \\
 & \leq &
 c_k\max_{e_{\vec{\epsilon}}\in \{0,\ldots,k\}} \text{U-}\limsup_{n_1,\ldots,n_k}\E_{n_1,\ldots,n_k} \Big|\int \prod_{\vec{\epsilon}\in \{0,1\}^k}\mathcal{C}^{|\epsilon|}T^{\sum_{i=1}^k[\epsilon_i n_i s]+e_{\vec{\epsilon}}}f\;d\mu \Big| \\ & \leq &
c_{k,s} \max_{e_{\vec{\epsilon}}\in \{0,\ldots,k\}}\text{U-}\limsup_{n_1,\ldots,n_k}\E_{n_1,\ldots,n_k} \Big|\int \prod_{\vec{\epsilon}\in \{0,1\}^k}\mathcal{C}^{|\epsilon|}T^{\vec{\epsilon}\cdot\vec{n}+e_{\vec{\epsilon}}}f\;d\mu \Big| \\
&\leq &
 c_{k,s} \max_{e_{\vec{\epsilon}}\in \{0,\ldots,k\}} \left(\text{U-}\limsup_{n_1,\ldots,n_k}\E_{n_1,\ldots,n_k} \Big|\int \prod_{\vec{\epsilon}\in \{0,1\}^k}\mathcal{C}^{|\epsilon|}T^{\vec{\epsilon}\cdot\vec{n}+e_{\vec{\epsilon}}}f\;d\mu \Big|^2\right)^{1/2} \\ & = &
 c_{k,s} \max_{e_{\vec{\epsilon}}\in \{0,\ldots,k\}} \left(\text{U-}\limsup_{n_1,\ldots,n_k}\E_{n_1,\ldots,n_k} \int \prod_{\vec{\epsilon}\in \{0,1\}^k}\mathcal{C}^{|\epsilon|}R^{\vec{\epsilon}\cdot\vec{n}+e_{\vec{\epsilon}}}F\;d(\mu\times\mu) \right)^{1/2},
\end{eqnarray*} using the Relation~(10) from \cite{HK99} and the Condition (1) of Lemma~3.9 from \cite{HK99}, this last term is bounded by
$$ c_{k,s} \max_{e_{\vec{\epsilon}}\in \{0,\ldots,k\}} \left(\prod_{\vec{\epsilon}\in \{0,1\}^k}\nnorm{R^{e_{\vec{\epsilon}}}F}_{k,\mu\times\mu,R} \right)^{1/2}=c_{k,s} \left(\nnorm{F}^{2^k}_{k,\mu\times\mu,R} \right)^{1/2}\leq c_{k,s} \nnorm{f}^{2^k}_{k+1,\mu,T},$$ where we have used the fact that $\nnorm{F}_{k,\mu\times\mu,R}=\nnorm{f\otimes\bar{f}}_{k,\mu\times\mu,T\times T}\leq \nnorm{f}^{2}_{k+1,\mu,T}.$
\end{proof}

In order to state the Proposition~\ref{P:F3} below, which is the analogous to Proposition~\ref{P:F8}, we need the following notation:

\medskip

Let $(X,\mathcal{X},\mu)$ be a probability space. If $(T_t)_{t\in\R}$ is a measure preserving flow and $p(t)=a_rt^r+\ldots+a_1t+a_0\in \R[t],$ we write 
\begin{equation}\label{E:FF}
T_{p(t)}=T^{t^r}_{a_r}\ldots T^{t}_{a_1} T_{a_0}.
\end{equation} 

We assign to the polynomial $p(t)$ an $(r+1)$-tuple of polynomials $\vec{p}=(p^{r}(t),\ldots,p^0(t)),$ where we put $p^i(t)=t^i$ if $a_i\neq 0$ and $p^i(t)=0$ otherwise. 

\medskip

We are now ready to define the notion of the $\R$-nice family.

\begin{definition*}
Let $\ell, m\in \N.$ The family of $m$ polynomial $\ell$-tuples of polynomials with maximum degree $d$ in $\R[t]$ $$(\mathcal{P}_1,\ldots,\mathcal{P}_\ell)=((p_{1,1},\ldots,p_{\ell,1}),\ldots,(p_{1,m},\ldots,p_{\ell,m}))$$ is called {\em $\R$-nice family} if the respective family of $m$ polynomial $ (d+1)\ell$-tuples $$ (\vec{\mathcal{P}}_1,\ldots,\vec{\mathcal{P}}_\ell):=((\vec{p}_{1,1},\ldots,\vec{p}_{\ell,1}),\ldots,(\vec{p}_{1,m},\ldots,\vec{p}_{\ell,m})),$$ where we complete the coordinates with zeros from the right in order every vector $\vec{p}_{i,j}$ to have $d+1$ coordinates, is a nice family (in $\Z[t]$).
\end{definition*}

We will now prove a proposition which implies Theorem~\ref{T:F2}.

\begin{proposition}\label{P:F3}
Let $\ell\in \N,$ $(X,\mathcal{X},\mu,T_1,\ldots,T_\ell)$ be a system, $(\mathcal{P}_1,\ldots,\mathcal{P}_\ell)$ be an $\R$-nice family of $\ell$-tuples of polynomials in $\R[t]$ with maximum degree $d$ and $f_1,\ldots,f_m\in L^\infty(\mu).$ Then there exists $k=k(d,\ell,m)\in\N$ such that if $\nnorm{f_1}_{k,\mu,T_1}=0,$ then the averages 
$$\frac{1}{N-M}\sum_{n=M}^{N-1}(\prod_{i=1}^\ell T^{[p_{i,1}(n)]}_i)f_1\cdot\ldots\cdot (\prod_{i=1}^\ell T^{[p_{i,m}(n)]}_i)f_m $$
converge to $0$ in $L^2(\mu)$ as $N-M\to\infty.$
\end{proposition}

\begin{proof}
We have to show, for the action $\displaystyle \prod_{i=1}^{\ell}T_{i,a_{i,1}}\cdot \ldots\cdot \prod_{i=1}^\ell T_{i,a_{i,m}},$ where $a_{i,j}=p_{i,j},$ and the $Y$-extensions of $f_j,$ $\hat{f_j},$ all defined in the proof of Theorem~\ref{T:M}, that the Condition (i) of Theorem~\ref{T:F4} holds and we will have the result, since we have already seen, in the proof of Theorem~\ref{T:3*}, that Condition~(ii) holds for real polynomials. According to the hypothesis, there exists $k=k(d,\ell,m)\in \N$ such that $\nnorm{f_1}_{k,\mu,T_1}=0.$ Using Lemma~\ref{L:F6} we have that $\nnorm{\hat{f}_1}_{k-1,\nu,S_1}=0,$ where, if $a^{1,1}_r$ is the leading coefficient of $p_{1,1},$ then $S_1$ can be chosen to be equal to $T_{1,|a^{1,1}_r|}.$ 

Writing each $T_{i,p_{i,j}}$ as in \eqref{E:FF} and using the fact that $(\mathcal{P}_1,\ldots,\mathcal{P}_\ell)$ is an $\R$-nice family in $\R[t]$, we get the respective nice family $(\vec{\mathcal{P}}_1,\ldots,\vec{\mathcal{P}}_\ell)$ in $\Z[t].$ 

The Condition (i)  of Theorem~\ref{T:F4} now follows from Proposition~\ref{P:F8}.
\end{proof}

Theorem~\ref{T:F2} will now follow from Proposition~\ref{P:F3}, as Corollary~5.2 in \cite{CFH} follows from Proposition~5.1 in \cite{CFH}.

\begin{proof}[Proof of Theorem~\ref{T:F2}]
We apply Proposition~\ref{P:F3} to the family $(\mathcal{P}_1,\ldots,\mathcal{P}_\ell)$ where $\mathcal{P}_1=(p_1,0,\ldots,0),$ $\mathcal{P}_2=(0,p_2,\ldots,0),\ldots, \mathcal{P}_\ell=(0,\ldots,0,p_\ell)$ which is an $\R$-nice family.
\end{proof}

Before we prove Corollary~\ref{C:F10} via Theorem~\ref{T:F2} and close this section, we recall the notion of a weak mixing system.

\begin{definition*}[\cite{FKO}]
Let $(X,\mathcal{X},\mu,T)$ be a system. If for any two functions $f,g\in L^2(\mu)$ we have that
$$\lim_{N\to\infty}\frac{1}{N}\sum_{n=1}^N\Big|\int f T^{n}g \;d\mu-\int f\;d\mu \int g\;d\mu\Big|=0,$$ then $T$ is called {\em weakly mixing transformation} and the system $(X,\mathcal{X},\mu,T)$ {\em weakly mixing system}.
\end{definition*}


It is an immediate consequence of the definition of the seminorms $\nnorm{\cdot}_{k}$ that, for weak mixing systems $(X,\mathcal{X},\mu,T),$ we have $\nnorm{f}_{k,\mu,T}=\Big| \int f\;d\mu \Big|$ for every $k\in\N.$

\begin{proof}[Proof of Corollary~\ref{C:F10}]
We can assume that the integral of some $f_i$ is $0,$ by using the elementary identity from \cite{FKO} $$\prod_{i=1}^\ell a_i-\prod_{i=1}^\ell b_i=\sum_{j=1}^k\left(\prod_{i=1}^{j-1}a_i \right)(a_j-b_j)\left(\prod_{i=j+1}^k b_i\right),$$ where we have set $\prod_{i=1}^0a_i=\prod_{i=k+1}^k b_i=1.$  We can actually assume that $\int f_1\;d\mu=0.$ Then $\nnorm{f_1}_{k,\mu,T_1}=0$ for all $k.$ The result now follows from Theorem~\ref{T:F2}.
\end{proof}

\section{Proof of main results}

 In this last section we give the proof of Theorems~\ref{T:1*}, ~\ref{T:s}, ~\ref{T:a} and ~\ref{T:F1}.

\medskip

Theorem~\ref{T:1*} will follow from Theorem~\ref{T:2*}, since, as we have already seen, every sequence of the form \eqref{E:integer} is $k$-regular (for all $k$) and $k$-weak-anti-uniform, for some $k$ depending on the integers $\ell, m$ and the maximum degree of the polynomials $p_{i,j}.$

 Theorem~\ref{T:s} will follow from Theorems~\ref{T:3*} and \ref{T:M2}, using also results from \cite{F}.
 
  Theorem~\ref{T:a} will follow from Theorem~\ref{T:s} and the analogous result, Theorem~1.4 of \cite{F}.

Finally, Theorem~\ref{T:F1} will follow from Theorem~\ref{T:F4} via Theorem~\ref{T:F7} and Lemma~\ref{L:F6}.

\begin{proof}[Proof of Theorem~\ref{T:1*}] Since any sequence $(a(n))$ of the form \eqref{E:integer} is $k$-weak-anti-uniform and $k$-regular, for some $k$ depending on the integers $\ell, m$ and the maximum degree of the polynomials $p_{i,j},$ as we showed in the end of Section 3 and Section 4, we can apply Theorem~\ref{T:2*} to get the conclusion.
\end{proof}

In order to prove Theorem~\ref{T:s}, we will make use of the following Hilbert-space variant of van der Corput's estimate
(see \cite{F}).

 \begin{lemma}\label{L:VDC}
Let  $(v_n)$ be a bounded  sequence of vectors in an inner product
space and $(I_N)$ be a sequence of intervals with lengths tending to infinity.
Then
$$
\limsup_{N\to\infty}
\norm{\frac{1}{|I_N|}\sum_{n\in I_N} v_n}^2\leq 4 \ \!
\limsup_{H\to\infty} \frac{1}{H}\sum_{h=1}^H
\limsup_{N\to\infty}\Big|
\frac{1}{|I_N|}\sum_{n\in I_N} \langle v_{n+h},v_{n}\rangle \Big|.
$$
\end{lemma}

\begin{proof}[Proof of Theorem~\ref{T:s}]
According to Theorem~\ref{T:2*}, in order to have the conclusion, we have to show that the sequence $\int f_0\cdot (\prod_{i=1}^\ell T_i^{[a_{i,1} n]})f_1\cdot \ldots \cdot (\prod_{i=1}^\ell T_i^{[a_{i,m} n]})f_m\, d\mu$ is $(m+1)$-regular and ($m+1$)-weak-anti-uniform.

The $(m+1)$-regularity follows from Section 3, making use of Proposition~\ref{P:nilkey} and the convergence result we get from Theorem~\ref{T:3*}.

 According to Theorem~\ref{T:M2}, since, as we saw in the proof of Theorem~\ref{T:3*}, real valued polynomials satisfy the Condition (ii) of this result,  if $S_{i,j}=T_{i,a_{i,j}}$ for every $i, j$ we have to prove that the sequence $\tilde{a}(n)=\int f_0\cdot (\prod_{i=1}^\ell S_{i,1}^{n})f_1\cdot \ldots \cdot (\prod_{i=1}^\ell S_{i,m}^{n})f_m\ d\mu$ is ($m+1$)-anti-uniform. This follows from the Subsection 2.3.1 of \cite{F} with induction, by successively applying Lemma~\ref{L:VDC}, using the van der Corput operation (at most) $m$ times in order the linear polynomial iterates that appear in $\tilde{a}(n)$ to be reduced into constant ones.    
\end{proof}


\begin{proof}[Proof of Theorem~\ref{T:a}]
 From the definitions, it is immediate that $\overline{\mathcal{B}_k}^{\norm{\cdot}_2}\subseteq\overline{\mathcal{C}_k}^{\norm{\cdot}_2}.$ From  Theorem~\ref{T:s}, we also have (for $\ell=m=k$) that $\overline{\mathcal{C}_k}^{\norm{\cdot}_2}\subseteq\overline{\mathcal{A}_k}^{\norm{\cdot}_2}.$ The result now follows, since $\overline{\mathcal{A}_k}^{\norm{\cdot}_2}=\overline{\mathcal{B}_k}^{\norm{\cdot}_2}$ (from Theorem~1.4 in \cite{F}). 
\end{proof}

We will close this article with the proof of Theorem~\ref{T:F1}.

\begin{proof}[Proof of Theorem~\ref{T:F1}]
Analogously to Proposition~\ref{P:F3}, we have to show for the action $ T_{1,a_1n^{r_1}}\cdot \ldots\cdot T_{\ell,a_\ell n^{r_\ell}}$ (we have set $a_{i,i}=a_i n^{r_i}$ and $a_{i,j}=0$ for all $i\neq j$) and the $Y$-extensions of $f_j,$ $\hat{f_j},$ all defined in the proof of Theorem~\ref{T:M}, that the Condition (i) of Theorem~\ref{T:F4} holds and we will have the result, since we have already seen that Condition~(ii) holds for real polynomials in the proof of Theorem~\ref{T:3*}. According to the hypothesis, there exists $k=k(d,\max r_i)\in \N$ such that $\nnorm{f_{i_0}}_{k,\mu,T_{i_0}}=0$ for some $1\leq i_0\leq \ell.$ Using Lemma~\ref{L:F6} we have that $\nnorm{\hat{f}_{i_0}}_{k-1,\nu,S_{i_0}}=0,$ where,  $S_{i_0}$ can be chosen to be equal to $T_{i_0,|a_{i_0}|}.$ 

Since $T_{i,a_in^{r_i}}=T^{n^{r_i}}_{i,a_i}$ and all $r_i$ are pairwise distinct, the Condition (i)  of Theorem~\ref{T:F4} follows from Theorem~\ref{T:F7}.
\end{proof}



\end{document}